\documentclass[fleqn,final]{zzz}
\usepackage{mathrsfs}
\usepackage{color}
\usepackage{tikz}
\usepackage{graphicx}
\usepackage{rotating}
\usepackage[pdftex]{hyperref}

\hypersetup{
   colorlinks = true,
   urlcolor = blue,
   linkcolor = red,
   citecolor = green
}

\newcommand{\hyp}[5]{\,\mbox{}_{#1}F_{#2}\!\left(
  \genfrac{}{}{0pt}{}{#3}{#4};#5\right)}

\newcommand{\qhyp}[5]{\,\mbox{}_{#1}\phi_{#2}\!\left(
\genfrac{}{}{0pt}{}{#3}{#4};#5\right)}

\newcommand{\myref}[1]{(\ref{#1})}
\def\cprime{$'$}

\newtheorem{thm}[lemma]{Theorem}
\newtheorem{cor}[lemma]{Corollary}

\makeatletter
\def\eqnarray{\stepcounter{equation}\let\@currentlabel=\theequation
\global\@eqnswtrue
\tabskip\@centering\let\\=\@eqncr
$$\halign to \displaywidth\bgroup\hfil\global\@eqcnt\z@
  $\displaystyle\tabskip\z@{##}$&\global\@eqcnt\@ne
  \hfil$\displaystyle{{}##{}}$\hfil
  &\global\@eqcnt\tw@ $\displaystyle{##}$\hfil
  \tabskip\@centering&\llap{##}\tabskip\z@\cr}

\def\endeqnarray{\@@eqncr\egroup
      \global\advance\c@equation\m@ne$$\global\@ignoretrue}

\def\@yeqncr{\@ifnextchar [{\@xeqncr}{\@xeqncr[5pt]}}
\makeatother

\newcommand{\bfx}{{\bf x}}
\newcommand{\bfxp}{{{\bf x}^\prime}}
\newcommand{\wbfx}{\widehat{\bf x}}
\newcommand{\wbfxp}{{\widehat{\bf x}^\prime}}
\newcommand{\Si}{{\mathbf S}}
\newcommand{\C}{{\mathbb C}}
\newcommand{\N}{{\mathbb N}}

\begin{document}

\renewcommand{\PaperNumber}{***}

\FirstPageHeading

\ShortArticleName{On a generalization of the Rogers generating function}

\ArticleName{On a generalization of the Rogers generating function}

\Author{Howard S. Cohl\,$^\dag\!\!\ $, Roberto 
S. Costas-Santos\,$^\S\!\!\ $, and Tanay Wakhare\,$^\ddag{}$}

\AuthorNameForHeading{H. S. Cohl, R. S. Costas-Santos, T. Wakhare}
\Address{$^\dag$ Applied and Computational 
Mathematics Division, National Institute of Standards 
and Technology, Mission Viejo, CA 92694, USA
\URLaddressD{
\href{http://www.nist.gov/itl/math/msg/howard-s-cohl.cfm}
{http://www.nist.gov/itl/math/msg/howard-s-cohl.cfm}
}
} 
\EmailD{howard.cohl@nist.gov} 

\Address{$^\S$ Departamento de F\'isica y Matem\'{a}ticas,
Universidad de Alcal\'{a},
c.p. 28871, Alcal\'{a} de Henares, Spain
} 
\URLaddressD{
\href{http://www.rscosan.com}
{http://www.rscosan.com}
}
\EmailD{rscosa@gmail.com} 
\Address{$^\ddag$ Department of Mathematics, University 
of Maryland, College Park, MD 20742, USA}
\EmailD{twakhare@gmail.com} 

\ArticleDates{Received 18 April 2018 in final form ????; 
Published online ????}

\Abstract{We derive a generalized Rogers generating 
function and corresponding definite integral, for 
the continuous $q$-ultraspherical polynomials by applying 
its connection relation and utilizing orthogonality. 
Using a recent generalization of the Rogers generating function
by Ismail \& Simeonov expanded in terms of Askey-Wilson
polynomials, we derive corresponding generalized expansions 
for the continuous $q$-Jacobi, and Wilson polynomials with two
and four free parameters respectively.  
Comparing the coefficients of the Askey-Wilson expansion to our 
continuous $q$-ultraspherical/Rogers expansion, we derive a 
new quadratic transformation for basic hypergeometric series connecting 
${}_2\phi_1$ and ${}_8\phi_7$.
}

\Classification{33C45, 05A15, 33C20, 34L10, 30E20}

\section{Introduction}

In the context of generalized hypergeometric orthgononal 
polynomials, H. Cohl developed in \cite[(2.1)]{CohlGenGegen} 
a series rearrangement technique which produces a 
generalization of the generating function for the 
Gegenbauer polynomials. We have since demonstrated 
that this technique is valid for a larger class of
hypergeometric orthogonal polynomials. For instance,
in \cite{Cohl12pow}, we applied this same technique 
to the Jacobi polynomials and in \cite{CohlMacKVolk}, 
we extended this technique to many generating functions 
for the Jacobi, Gegenbauer, Laguerre, and Wilson polynomials.

The series rearrangement technique combines a connection 
relation with a generating function, resulting in a series 
with multiple sums. The order of summations are then 
rearranged and the result often simplifies to produce a
generalized generating function whose coefficients are 
given in terms of
generalized or basic hypergeometric functions. This 
technique is especially productive when using connection 
relations with one free parameter, since the relation is
most often a product of Pochhammer and $q$-Pochhammer 
symbols.

Basic hypergeometric orthogonal polynomials with more 
than one free parameter, such the Askey-Wilson polynomials, 
have multi-parameter connection relations. 
These connection relations are in general given by single 
or multiple summation expressions. For the Askey-Wilson
polynomials, the connection relation with four free 
parameters is given as a basic double hypergeometric series.
The fact that the four free parameter connection coefficient 
for the Askey-Wilson polynomials is given by a double 
sum was known to Askey and Wilson as far back as 1985 (see 
\cite[p. 444]{Ismail}). When our series rearrangement 
technique is applied to cases with more than one free 
parameter, the resulting coefficients of the generalized
generating function are rarely given in terms of a basic 
hypergeometric series. The more general problem of 
generalized generating functions with more than one free 
parameter requires the theory of multiple basic 
hypergeometric series and is not treated in this paper.

The coefficients of our derived generalized 
generating functions are basic hypergometric 
functions.  There are many known summation 
formulae for basic hypergeometric functions 
(see for instance, \cite[Sections 17.5--17.7]{NIST:DLMF}). 
One could study the special combinations 
of parameters which allow for summability of 
our basic hypergeometric coefficients. 
However, in these cases the affect of summability 
for special parameter values, simply reduces to 
a re-expression of the original generating function 
used to generate the generalizations that we 
present. So therefore nothing new would then 
be learned by this study.

In this paper, we apply this technique to 
generalize generating functions for basic 
hypergeometric orthogonal polynomials in the 
$q$-analog of the Askey scheme 
\cite[Chapter 14]{Koekoeketal}.
These are the Askey-Wilson polynomials 
(Section \ref{AskeyWilsonpolynomials}),
the continuous $q$-Jacobi polynomials 
(Section \ref{continuousqJacobipolynomials}),
and the continuous  $q$-ultraspherical/Rogers polynomials
(Section \ref{Rogerscontinuousqultrasphericalpolynomials}).
In Section 
\ref{Definiteintegralsinfiniteseriesandqintegrals},
we have also computed new definite integrals 
corresponding to our generalized generating function 
expansions using orthogonality for the studied basic 
hypergeometric orthogonal polynomials.


In addition of being of independent interest, 
this investigation was motivated by an 
application of generalized generating functions 
in the non-$q$ regime
\cite{Cohl12pow, CohlGenGegen}. This would be 
the generation of $q$-polyspherical addition 
theorems in terms of a product of $q$-zonal 
harmonics.  In order to compute these 
$q$-analogues, one would need to derive a 
$q$-analogue of the addition theorem for
the hyperspherical harmonics (see \cite{WenAvery}; 
see also \cite[Section 10.2.1]{FanoRau})
\[
C_n^{d/2-1}(\cos\gamma)
=\frac{2(d-2)\pi^{d/2}}
{(2n+d-2)\Gamma(d/2)}
\sum_{K} Y_n^K (\wbfx) \overline{Y_n^K (\wbfxp)},
\]
\noindent 
where, for a given value of ${n}\in\mathbb N_0:=\{0,1,2,\ldots\}$,
$C_n^\mu$ is the Gegenbauer polynomial, $K$ stands for a 
set of $(d-2)$-quantum numbers identifying normalized 
hyperspherical harmonics $Y_{{n}}^{K}:\Si^{d-1}\to\mathbb C$, 
and $\gamma$ is separation angle between two arbitrary 
vectors  $\bfx,\bfxp\in\mathbb R^d$.
One would also need $q$-analogues of a fundamental solution of
the polyharmonic equation, and Laplace's expansion
\[
\frac{1}{\|\bfx-\bfxp\|^{d-2}}=\sum_{l=0}^\infty 
\frac{r_<^l}{r_>^{l+d-2}}
C_l^{d/2-1}(\cos\gamma),
\]
which is the $q\uparrow 1^{-}$ limit of the generating 
function for the continuous $q$-ultraspherical/Rogers 
polynomials (see (\ref{4:22}) below).
These analogues do not exist in the literature, 
however they may be found by using material from 
\cite[Section 3]{Noumietal96},  \cite{IorgovKlimyk2001}, 
which we will attempt in future publications.
Addition theorems for continuous $q$-ultraspherical/Rogers polynomials
should also be useful here \cite{Koorwinder2018}.


\section{Preliminaries}
Throughout the paper, we adopt the following notation to indicate 
sequential positive and negative elements, in a list of elements, namely
\[
\pm a:=\{a,-a\}.
\]
If the symbol $\pm$ appears in an expression, but not in 
a list, it is to be treated as normal. In order to obtain 
our derived identities, we rely on properties of the
$q$-Pochhammer symbol, also called $q$-shifted factorial, 
which is defined for $n\in\mathbb N_0$ as follows:
\begin{equation}
\label{2:1}
(a;q)_0:=1,
\quad
(a;q)_n:=(1-a)(1-aq)\cdots(1-aq^{n-1}),
\end{equation}
and
\begin{equation}
(a;q)_\infty:=\prod_{n=0}^\infty (1-aq^{n}),
\label{2:2}
\end{equation}
where $q, a\in \mathbb C$, $0<|q|<1$.
We define the $q$-factorial as \cite[(1.2.44)]{GaspRah}
\[
[0]_q!:=1, \ [n]_q!:=[1]_q[2]_q\cdots 
[n]_q, \ n\ge 1,
\]
where the $q$-number is defined as 
\cite[(1.8.1)]{Koekoeketal}
\[
[z]_q:=\frac{1-q^z}{1-q},\qquad z
\in\mathbb C.
\]
Note that $[n]_q! = (q;q)_n / (1-q)^n$.

The following properties for the $q$-Pochhammer 
symbol can be found in Koekoek et al. (2010) 
\cite [(1.8.7), (1.8.10-11), (1.8.14), 
(1.8.19), (1.8.21-22)]{Koekoeketal},
namely for appropriate values of $a$ and 
$k\in\mathbb N_0$,
\begin{eqnarray}
\label{2:3} &&\hspace{-6.5cm}
(a;q)_{n+k}=(a;q)_k (aq^k;q)_n = (a;q)_n
(aq^n;q)_k,\\[0.2cm]
&&\hspace{-6.5cm}(a^2;q^2)_n=(\pm a;q)_n. 
\label{2:4} 
\end{eqnarray}
Observe that by using (\ref{2:1}) and 
(\ref{2:4}), we get
\begin{eqnarray}
&&\hspace{-6.5cm}(aq^n;q)_n=\frac{(\pm 
\sqrt{a},\pm \sqrt{aq};q)_n}{(a;q)_n}. 
\label{2:5}
\end{eqnarray}

\begin{lemma}
Let $n\in\mathbb N_0$, $q, a,\beta\in\mathbb C$, $0<|q|<1$. 
Then
\begin{equation}
\label{2:6}
(a;q)_{n+\beta}=(a;q)_n(aq^n;q)_\beta.
\end{equation}
\end{lemma}
\begin{proof}
Follows from the identity \cite[(1.8.9)]{Koekoeketal}
\begin{equation}
(a;q)_\beta:=
\dfrac{(a;q)_\infty}{(aq^\beta;q)_\infty},
\label{2:7}
\end{equation}
where $a\in\mathbb C$, and \cite[(1.8.8)]{Koekoeketal} 
$(a;q)_n=(a;q)_\infty/(aq^n;q)_\infty$.
\end{proof}

\begin{lemma}
Let $q, \alpha,\beta\in\mathbb C$, $0<|q|<1$. Then
\begin{equation}
\label{2:8}
\lim_{q\uparrow1^{-}}\frac{(q^\alpha;q)_\beta}
{(1-q)^\beta}=(\alpha)_\beta.
\end{equation}
\end{lemma}
\begin{proof}
Define the $q$-gamma function $\Gamma_q$ by 
\cite[(1.9.1)]{Koekoeketal}
\[
\Gamma_q(x):=\frac{(1-q)^{1-x}(q;q)_\infty}
{(q^x;q)_\infty},
\]
and the arbitrary $q$-Pochhammer symbol by 
(\ref{2:7}).
Observe that, by using (\ref{2:6}), 
if $\Re \beta<0$ then
\begin{equation}
(a;q)_\beta:=\frac 1{(a q^{\beta};q)_{-\beta}}.
\label{2:9}
\end{equation}
Taking the previous expressions we have 
that the arbitrary Pochhammer symbol for 
$\Re \beta>0$ is defined naturally by
\[
(\alpha)_\beta:=\frac{\Gamma(\alpha+\beta)}
{\Gamma(\alpha)},
\]
and if $\Re \beta<0$ then
$(\alpha)_\beta:=1/(\alpha+\beta)_{-\beta}$.
Then
\begin{itemize}
\item If $\alpha+\beta\in -\mathbb N_0$ then 
the result is straightforward by definition
since $(-n)_n=0$ and $(q^{-n};q)_n=0$ for any 
$n\in \mathbb N_0$.
\item If $\Re \beta>0$ then
\[
\lim_{q\uparrow1^{-}}
\frac{(q^\alpha;q)_\beta}{(1-q)^\beta}
=\lim_{q\uparrow1^{-}}
\frac{(q^\alpha;q)_\infty}{(1-q)^\beta(q^{\alpha
+\beta};q)_\infty}=\lim_{q\uparrow1^{-}}
\frac{\Gamma_q(\alpha+\beta)}{\Gamma_q(\alpha)}
=(\alpha)_\beta,
\]
since \cite[Section 1.9]{Koekoeketal} $\lim_{q\uparrow1^{-}}\Gamma_q(x)=\Gamma(x)$.
\item If $\Re \beta<0$ then 
\[
\lim_{q\uparrow1^{-}}
\frac{(q^\alpha;q)_\beta}{(1-q)^\beta}
\stackrel{(\ref{2:9})}=\lim_{q\uparrow1^{-}}
\frac{(1-q)^{-\beta}}
{(q^{\alpha+\beta};q)_{-\beta}}
=\lim_{q\uparrow1^{-}}\frac{(1-q)^{-\beta}
(q^\alpha;q)_\infty}{(q^{\alpha+\beta}
;q)_\infty}=\lim_{q\uparrow1^{-}}\frac
{\Gamma_q(\alpha+\beta)}{\Gamma_q(\alpha)}
=(\alpha)_\beta.
\]
\end{itemize}
This completes the proof.
\end{proof}

We also take advantage of the $q$-binomial theorem 
\cite[(1.11.1)]{Koekoeketal}
\begin{equation}
\qhyp10a-{q,z}=\frac{(az;q)_\infty}{(z;q)_\infty},
\quad |z|<1, \label{2:10}
\end{equation}
where we have used \myref{2:2}.
The basic hypergeometric series, which we 
often use, is defined as \cite[(1.10.1)]{Koekoeketal}
\begin{equation}
{}_r\phi_s\left(
\begin{array}{c}
a_1,\ldots,a_r\\
b_1,\ldots,b_s
\end{array}
;q,z
\right):=\sum_{k=0}^\infty \frac{(a_1,\ldots,a_r;q)_k}
{(q,b_1,\ldots,b_s;q)_k} \left((-1)^kq^{\binom k 2}
\right)^{1+s-r}z^k.
\label{2:11}
\end{equation}

Let us prove some inequalities that we later use.
\begin{lemma} 
Let $j\in \mathbb N$, $k,n\in \mathbb N_0$, $z\in 
\mathbb C$, $\Re u>0$, $v\ge 0$,   and $0<|q|<1$. 
Then
\begin{eqnarray}
\label{2:12} \dfrac{(q^u;q)_j}{(1-q)^j}&\ge& [\Re u]_q 
[j-1]_q!, \\
\label{2.13}\dfrac{(q^u;q)_n}{(q;q)_n}&\le&[1+n]_q^ u,\\
\label{2:14}\dfrac{(q^{v+k};q)_n}{(q^{u+k};q)_n}&\le& 
\frac{[n+1]_q^{v+1}}{[\Re (u)]_q}.
\end{eqnarray}
\end{lemma}
\begin{proof}
If $0<|q|<1$ then

\[
\dfrac{(q^u;q)_j}{(1-q)^j}=\prod_{k=1}^{j-1} 
\frac{1-q^{u+k-1}}{1-q}\ge \frac{1-q^u}{1-q}
\prod_{k=1}^{j-1} \frac{1-q^{k}}{1-q}\ge [\Re(u)]_q 
[j-1]_q!.
\]
This completes the proof of \eqref{2:12}. Choose 
$m\in \mathbb N_0$ such that $m\le u \le m + 1$.  
Then $q^{m+1}\le q^u$, so
\[
\dfrac{(q^u;q)_n}{(q;q)_n}=\prod_{k=0}^{n-1} 
\frac{1-q^{u+k}}{1-q^{1+k}}\le \prod_{k=1}^n 
\frac{1-q^{m+k}}{1-q^{k}}
=\prod_{k=1}^m \frac{1-q^{n+k}}{1-q^{k}}
\le [n+1]_q^m\le [n+1]_q^{u}.
\]
This completes the proof of \eqref{2.13}. Without loss 
of generality we assume $u>0$. If $v\le  u$ then the 
inequality is clear, so let us assume
that $0<u<v$. Since $0<|q|<1$ and for $t\ge 0$,
\[
\frac{t+v}{t+u}\le \frac v u,
\]
we have
\[
\dfrac{(q^{v+k};q)_n}{(q^{u+k};q)_n}\le\frac{(q^v)_n}
{(q^u)_n}\le\frac 1{[u]_q}\frac{(q^v)_n}{[n-1]_q!(1-q)^n}.
\]
Choose $m\in \mathbb N$ so that $m-1<v\le m$. Then
\[
\dfrac{(q^{v+k};q)_n}{(q^{u+k};q)_n}\le
\frac 1{[u]_q}\frac{[n]_q(q^m;q)_n}{(q;q)_n}=
\frac 1{[u]_q}\frac{[n]_q(q^n;q)_{m-1}}{(q;q)_{m-1}}\le
\frac 1{[u]_q}[n]_q [n+1]_q^{m-1}\le
\frac 1{[u]_q} [n+1]_q^{v+1}.
\]
This completes the proof of \eqref{2:14}. 
Therefore all the previous formulae hold true.
\end{proof}


\section{The Askey-Wilson polynomials}
\label{AskeyWilsonpolynomials}

The Askey-Wilson polynomials are defined as
\cite[(14.1.1)]{Koekoeketal}
\[
p_n(x;{\bf a}|q):=a_1^{-n}(a_1a_2,a_1a_3,a_1a_4;q)_n\,
\qhyp43{q^{-n},a_1a_2a_3a_4q^{n-1},a_1e^{i\theta},
	a_1e^{-i\theta}}{a_1a_2,a_1a_3,a_1a_4}{q,q},
\quad x=\cos\theta.
\]

\vspace{0.2cm}
\noindent In \cite[Theorem 4.2]{IsmailSimeonov2015} 
the following Askey-Wilson polynomial 
expansion of the Rogers generating function
\cite[(14.10.27)]{Koekoeketal} is proven.

\begin{thm}[Ismail \& Simeonov (2015)]
Let $t,\beta,q\in\C$, 
$\max\{|a_1|, |a_2|, |a_3|, |a_4|, |t|, |q|\}<1$, ${\bf a}:=\{a_1,a_2,a_3,a_4\}$,
$x\in[-1,1]$. Then
\begin{equation}
\frac{(t\beta e^{i\theta},t\beta e^{-i\theta};q)_\infty}
{(te^{i\theta},te^{-i\theta};q)_\infty}
=\sum_{n=0}^\infty c_n(\beta ,t,{\bf a};q)p_n(x;{\bf a}|q),
\label{3:15}
\end{equation}
where
\begin{eqnarray}
&&c_n(\beta ,t,{\bf a};q):=
\frac
{t^n(\beta ;q)_n(q^na_1\beta t,q^na_2\beta t
,q^na_3\beta t,q^na_1a_2a_3t;q)_\infty}
{(q,q^{n-1}a_1a_2a_3a_4;q)_n(a_1t,a_2t,a_3t
,q^{2n}a_1a_2a_3\beta t;q)_\infty}\nonumber\\
&&\hspace{1cm}\times\qhyp87{q^{2n-1}a_1a_2a_3
\beta t,\pm q^{n+\frac12}(a_1a_2a_3\beta 
t)^\frac12,q^na_1a_2,q^na_1a_3, q^na_2a_3,
\beta t/a_4,q^n \beta}{\pm q^{n-\frac12}
(a_1a_2a_3\beta t)^\frac12,q^na_1\beta t
,q^na_2\beta t,q^na_3\beta t,q^{2n}a_1a_2a_3a_4,
q^na_1a_2a_3t}{q,a_4t}.\nonumber
\end{eqnarray}
\end{thm}

\begin{remark}
Note (\ref{3:15}) is a generalization of the Rogers generating function
(the generating function where the coefficient multiplying $t^n$ is unity) 
\cite[(14.10.27)]{Koekoeketal}
\begin{equation}
\frac{(t\beta e^{i\theta},t\beta e^{-i\theta};q)_\infty}
{(te^{i\theta},te^{-i\theta};q)_\infty}
=\sum_{n=0}^\infty C_n(x;\beta|q)t^n, \quad x=\cos\theta,
\label{4:22}
\end{equation}
where $C_n(x;\beta|q)$ is the continuous $q$-ultraspherical/Rogers polynomial 
(see Section 
\ref{Rogerscontinuousqultrasphericalpolynomials}
below).
\end{remark}

\begin{remark}
Note that to compute such basic hypergeometric functions,
it is convenient to use (\ref{2:5}).
\end{remark}



\subsection{The Wilson limit for the 
Ismail-Simeonov Rogers-type generalized 
generating function}

In this section we obtain by starting from the generalized generating
function given in \cite{IsmailSimeonov2015}, a new infinite series identity 
over the Wilson polynomials. Define 
$a_{12}:=a_1+a_2$,
$a_{13}:=a_1+a_3$,
$a_{23}:=a_2+a_3$,
$a_{123}:=a_1+a_2+a_3$,
$a_{1234}:=a_1+a_2+a_3+a_4$.

\begin{thm} Let $a_1, a_2, a_3, a_4, u, x, t \in \mathbb C$, 
$t\pm ix\not\in-\mathbb N_0$. Then
\begin{eqnarray}
&&\frac{\Gamma(t+ix)\Gamma(t-ix)}{\Gamma(u+ix)\Gamma(u-ix)}
=\frac{(a_{123})_u(a_1,a_2,a_3)_t}
{(a_{123})_t(a_1,a_2,a_3)_u}
\sum_{n=0}^\infty 
\frac{(u-t,a_{1234}-1)_n(a_{123}+u)_{2n}
}
{n!(a_1+u,a_2+u,a_3+u,a_{123}+t)_n(a_{1234}-1)_{2n}} \nonumber\\
&&\hspace{0.4cm}\times {}_7F_6\left(\begin{array}{c} \lambda, 1+\lambda/2, a_{12}+n, 
a_{13}+n, a_{23}+n, u-a_4, u-t+n \\ \lambda/2, a_1+u+n, a_2+u+n,
a_3+u+n, 
a_{123}+t+n,
a_{1234}+2n 
\end{array};1\right)
W_n(x^2;{\bf a})
,
\label{Wilsonlimit}
\end{eqnarray}
where $\lambda:=2n-1+a_{123}+u$.
\end{thm}
\begin{proof}
We begin with the Ismail-Simeonov expansion 
\cite[(4.9)]{IsmailSimeonov2015}
\[
\frac
{(ue^{i\theta},ue^{-i\theta};q)_\infty}
{(te^{i\theta},te^{-i\theta};q)_\infty}
=\sum_{n=0}^\infty c_n(u,t,{\bf a})p_n(x;{\bf a}|q),
\]
where
\begin{eqnarray*}
&&c_n(u,t,{\bf a})=\frac
{t^n(u/t;q)_n(q^na_1u,q^na_2u,q^na_3u,q^na_1a_2a_3t;q)_\infty}
{(q,q^{n-1}a_1a_2a_3a_4;q)_n(a_1t,a_2t,a_3t,qb_1;q)_\infty}\\
&&\hspace{1cm}\times\qhyp87
{b_1,\pm qb_1^\frac12,q^na_1a_2,q^na_1a_3,q^na_2a_3,u/a_4,q^nu/t}
{\pm b_1^\frac12,q^na_1u,q^na_2u,q^na_3u,q^{2n}a_1a_2a_3a_4,q^na_1a_2a_3t}
{q,a_4t},
\end{eqnarray*}
with $b_1=q^{2n-1}a_1a_2a_3u$.

We apply the substitutions $a_k\mapsto q^{a_k}$, for all
$k\in\{1,2,3,4\}$, $e^{i\theta}\mapsto q^{ix},$
$t\mapsto q^t$, $u\mapsto q^u$, multiply both sides by $(1-q)^{2(u-t)}$
and take the limit as $q\uparrow 1^{-}$. We use 
(\ref{2:6}), 
(\ref{2:7}), 
(\ref{2:8}),
and apply the relation
\cite[(14.1.21)]{Koekoeketal} 
\[
\lim_{q\uparrow 1^{-}} \frac{p_n((q^{ix}+q^{-ix})/2;q^{a_1},q^{a_2},q^{a_3},q^{a_4}|q)}
{(1-q)^{3n}}=W_n(x^2;{\bf a}),
\] 
where $W_n$ is the Wilson polynomial 
\cite[(9.1.1)]{Koekoeketal}, and ${\bf a}:=\{a_1,a_2,a_3,a_4\}$. 
Since
\[
\frac{1}{(t+ix,t-ix)_{u-t}}=
\frac{\Gamma(t+ix)\Gamma(t-ix)}{\Gamma(u+ix)\Gamma(u-ix)},
\]
the theorem holds.
\end{proof}
Note that this formula can easily be hand-verified 
in the regime $t-u\in\mathbb N_0$ since in this case the 
sum terminates.


\begin{remark}
Observe that the generalized hypergeometric function ${}_7F_6$ in 
(\ref{Wilsonlimit}) is actually a Wilson function, a generalization
of the Wilson polynomial (see for instance, \cite[(1.3)]{Groenevelt2003}).
This can be demonstrated using Bailey's $W$ notation
for a very-well poised ${}_7F_6$ of argument unity
\[
W(a;b,c,d,e,f):=
\hyp76{a,\frac{a}{2}+1,b,c,d,e,f}{\frac{a}{2},
1+a-b,
1+a-c,
1+a-d,
1+a-f}{1}.
\]
In our case, the ${}_7F_6$ can be written as 
\[
W(\lambda;a_{12}+u,a_{13}+u,a_{23}+u,u-a_4,u-t+n).
\]
\end{remark}

\section{The continuous $q$-Jacobi polynomials}
\label{continuousqJacobipolynomials}

In a correspondence with the non-$q$ case, we would 
like to examine consequences of the  specialization 
of the above formula in terms of the continuous 
$q$-Jacobi polynomials and the continuous 
$q$-ultraspherical/Rogers polynomials. 
For the continuous $q$-Jacobi polynomials, 
we adopt the standard normalization adopted by Rahman et al.
in \cite[(14.10.1)]{Koekoeketal}. However, in order to 
simplify our formulae we have further replaced
$q^{\alpha+\frac12},q^{\gamma+\frac12}\mapsto\alpha,
\gamma$. Using this notation one has
\begin{eqnarray}
P_n^{(\alpha,\gamma)}(x|q)&:=&\frac{\alpha^{n/2}}
{(q,-(\alpha\gamma)^\frac12,-(q\alpha
	\gamma)^\frac12;q)_n} p_n\left(x;\alpha^\frac12
,-\gamma^\frac12,-(q\gamma)^\frac12,(q\alpha)^\frac12
|q\right)\nonumber\\ &=&\frac{(q^\frac12\alpha;q)_n}
{(q;q)_n}\qhyp43{q^{-n},q^n\alpha\gamma,\alpha^\frac12
	e^{i\theta},\alpha^\frac12e^{-i\theta}}
{q^\frac12\alpha,-(\alpha\gamma)^\frac12,-(q\alpha
	\gamma)^\frac12}{q,q}.
\label{3:16}
\end{eqnarray}
Note that some consequences of this notation are
\[
C_n(x;\beta|q)=
\beta^{-\frac{n}{2}}
\frac
{(\beta^2;q)_n}
{(q^\frac12\beta;q)_n}
P_n^{(\beta,\beta)}(x|q),
\]
and
\begin{equation}
\lim_{q\uparrow1^{-}}P_n^{\left(q^{\alpha+1/2},
	q^{\gamma+1/2}\right)}(x|q)=P_n^{(\alpha,\gamma)}(x),
\label{3:17}
\end{equation}
where
$P_n^{(\alpha,\gamma)}$ is the Jacobi polynomial 
\cite[(18.5.7)]{NIST:DLMF}. Using 
(\ref{3:15}), (\ref{3:16}), 
we obtain
\begin{equation}
\frac{(t\beta e^{i\theta},t\beta e^{-i\theta}
	;q)_\infty}{(t e^{i\theta},t e^{-i\theta}
	;q)_\infty}=\sum_{n=0}^\infty d_n(\beta,t
,\alpha,\gamma,q)P_n^{(\alpha,\gamma)}(x|q),
\label{3:18}
\end{equation}
where
{\small \begin{eqnarray}
&&
d_n(\beta,t,\alpha,\gamma,q):=
\frac{(t\alpha^{-\frac12})^n
(\beta,-(\alpha\gamma)^{\frac12},-(q\alpha
\gamma)^{\frac12};q)_n}{(q^n\alpha\gamma;q)_n}
\frac
{(q^n\alpha^{\frac12}\beta t,-q^n\gamma^{\frac12}\beta t,-q^{n+\frac12}\gamma^{\frac12}\beta t,
q^{n+\frac12}\alpha^{\frac12}\gamma t;q)_\infty
}
{
(\alpha^{\frac12}t,-\gamma^{\frac12}t
,-(q\gamma)^{\frac12}t,q^{2n+\frac12}
\alpha^{\frac12}\gamma\beta t;q)_\infty}\nonumber\\
&&\hspace{0cm}\times
\qhyp87
{q^{2n-\frac12}\alpha^\frac12\gamma\beta t,
\pm q^{n+\frac34}\alpha^\frac14(\gamma\beta t)^\frac12,
-q^n(\alpha\gamma)^\frac12,
-q^{n+\frac12}(\alpha\gamma)^\frac12,
q^{n+\frac12}\gamma,
(q\alpha)^{-\frac12}\beta t,
q^n\beta}
{\pm q^{n-\frac14}\alpha^\frac14(\gamma\beta t)^\frac12,
q^n\alpha^\frac12\beta t,
-q^n\gamma^\frac12\beta t,
-q^{n+\frac12}\gamma^\frac12\beta t,
q^{2n+1}\alpha\gamma,
q^{n+\frac12}\alpha^\frac12\gamma t
}{q,(q\alpha)^{\frac12}t}.\nonumber
\end{eqnarray}}
Using (\ref{3:17}) in (\ref{3:18}) we obtain a generalization of the
Gegenbauer generating function 
\[
\frac{1}{(1+t^2-2tx)^\beta}=\sum_{n=0}^\infty
\frac{t^n(\beta)_n(\alpha+\gamma+1)_nP_n^{(\alpha,\gamma)}(x)}
{\left(\frac{\alpha+\gamma+1}{2}\right)_n
\left(\frac{\alpha+\gamma+2}{2}\right)_n(1+t)^{2(n+\beta)}}
\hyp21{\gamma+n+1,n+\beta}{2n+\alpha+\gamma+2}{\frac{4t}{(1+t)^2}},
\]
which is equivalent to \cite[(3.1)]{Cohl12pow}
\begin{eqnarray}
&&\frac{1}{(z-x)^\nu}=
\frac{(z-1)^{\alpha+1-\nu}(z+1)^{\beta+1-\nu}}
{2^{\alpha+\beta+1-\nu}}\nonumber\\
&&\hspace{1cm}\times\sum_{n=0}^\infty 
\frac{(2n+\alpha+\beta+1)\Gamma(\alpha+\beta+n+1)(\nu)_n}
{\Gamma(\alpha+n+1)\Gamma(\beta+n+1)}
Q_{n+\nu-1}^{(\alpha+1-\nu,\beta+1-\nu)}(z) 
P_n^{(\alpha,\beta)}(x),
\label{3:19}
\end{eqnarray}
where $z=(t+t^{-1})/2$, and $Q_\nu^{(\alpha,\gamma)}$ is a Jacobi 
function of the second kind.
The $q$-analogue of the specialization of (\ref{3:19}) with $\nu=1$
\cite[(9.2.1)]{Szego}
\begin{eqnarray}
&&\frac{1}{z-x}=
\frac{(z-1)^\alpha(z+1)^\beta}{2^{\alpha+\beta}}
\sum_{n=0}^\infty 
\frac{(2n+\alpha+\beta+1)\Gamma(\alpha+\beta+n+1)n!}
{\Gamma(\alpha+1+n)\Gamma(\beta+1+n)}
Q_n^{(\alpha,\beta)}(z) 
P_n^{(\alpha,\beta)}(x),\nonumber
\end{eqnarray}
is (\ref{3:18}) with $\beta=q$.

\begin{remark}
Note that the functions $x\mapsto (2t)^{-1}(1+t^2-2tx)$ and
$x\mapsto z-x$ are linearly related by using the Szeg\H{o} transformation
\[
z=\frac{t+t^{-1}}{2}.  
\]
We use this transformation regularly in various identities.
\end{remark}

\section{The continuous $q$-ultraspherical/Rogers polynomials}
\label{Rogerscontinuousqultrasphericalpolynomials}

The continuous $q$-ultraspherical/Rogers polynomials 
are defined as \cite[(14.10.17)]{Koekoeketal}
\[
C_n(x;\beta|q):=\frac{(\beta;q)_n}{(q;q)_n}e^{in\theta}\,
{}_2\phi_1\left(
\begin{array}{c}
q^{-n},\beta\\[0.2cm]
\beta^{-1}q^{1-n}
\end{array};q,q\beta^{-1}e^{-2i\theta}
\right), \quad x=\cos \theta.
\]
By starting with generating functions for the 
continuous $q$-ultraspherical/Rogers polynomials \cite[(14.10.27--33)]{Koekoeketal}, we derive 
generalizations using the connection relation 
for these polynomials, namely
\cite[(13.3.1)]{Ismail}
\begin{equation}
\label{4:20}
C_{n}(x;\beta\,|\,q) =\frac{1}{1-\gamma}
\sum_{k=0}^{\lfloor n/2 \rfloor}\frac{
(1-\gamma q^{n-2k})\gamma^{k}(\beta
\gamma^{-1};q)_{k}(\beta;q)_{n-k}}{(q;q)_{k}
(q\gamma;q)_{n-k}}\,C_{n-2k}(x;\gamma|q).
\end{equation}

\begin{thm} \label{theo:4}
Let $x\in[-1,1]$, $|t|<1$, $\beta,\gamma\in 
(-1,1)\setminus\{0\}$, $0<|q|<1$.
Then
\begin{equation}
\hspace{-0.3cm}\frac{(t\beta e^{i\theta}
,t\beta e^{-i\theta};q)_\infty}
{(t e^{i\theta},t e^{-i\theta};q)_\infty}
=\sum_{n=0}^\infty
\frac{(\beta;q)_n}{(\gamma;q)_n}
\,{}_2\phi_1\left(
\begin{array}{c}
\beta\gamma^{-1},\beta q^n\\
\gamma q^{n+1}\end{array}
;q,\gamma t^2\right)
C_n(x;\gamma|q)\, t^n.
\label{4:21}
\end{equation}
\end{thm}
\begin{proof}
The generating function for the continuous $q$-ultraspherical/Rogers 
polynomials (\ref{4:22})
is a $q$-analogue of the generating function for the Gegenbauer polynomials
\cite{Gegenbauer1874} 
\begin{equation}
\frac{1}{(1+t^2-2tx)^\mu}=\sum_{n=0}^\infty t^n 
C_n^\mu(x).
\label{4:23}
\end{equation}
The proof follows as above by starting with 
(\ref{4:22}), inserting (\ref{4:20}), 
shifting the $n$ index by $2k$, reversing the
order of summation. We use \myref{2:3} 
through \myref{2:11},
since $|a_n|=|t|^n$, $|c_{n,k}|\le K_6 
[n-k+1]^{\sigma_3}$, $|C_n(x;\beta|q)|\le 
[n+1]^{\sigma_4}/|1-n+\log_q \beta |$, 
therefore for $n$ large enough
\begin{equation}
\label{4:24}
|C_n(x;\beta|q)|\le  [n+1]^{\sigma_4} \le 
(n+1)^{\sigma_4},
\end{equation}
where
$K_6=1/(|1-\gamma|\log_q \gamma)$, 
$\log_qz:=\log z/\log q,$
and $\sigma_3$ and $\sigma_4$ are
independent of $n$. Therefore since
\[
\sum_{n=0}^\infty |a_n|\sum_{k=0}^{\lfloor 
n/2 \rfloor} |c_{k,n}| |C_k(x;\beta|q)|\le
K_7 \sum_{n=0}^\infty |t|^n (n+1)^{\sigma_5}<\infty,
\]
the result is proven.
\end{proof}

\noindent The above result (\ref{4:21}) 
is a $q$-analogue of 
\cite[Theorem 2.1]{CohlGenGegen}
\begin{equation}
\frac{1}{(z-x)^\nu}=
\frac{2^{\mu+\frac12}\Gamma(\mu)e^{i\pi(\mu
-\nu+\frac12)}}{\sqrt{\pi}\,\Gamma(\nu)
(z^2-1)^{\frac{\nu-\mu}{2}-\frac14}}
\sum_{n=0}^\infty(n+\mu) Q_{n+\mu-
\frac12}^{\nu-\mu-\frac12}(z)C_n^\mu(x).
\label{4:25}
\end{equation}
\begin{remark}
Note that the Gegenbauer polynomials are generalized by the Jacobi polynomials using 
\cite[(9.8.19)]{Koekoeketal} 
\begin{equation}
C_n^\mu(x)=\frac{(2\mu)_n}{(\mu+\frac12)_n}P_n^{(\mu-\frac12,\mu-\frac12)}(x).
\label{GegenJacobi}
\end{equation}
\end{remark}

By using Theorem \ref{theo:4} as a starting 
point, there are a number of interesting
results which follow. 
The first is an expansion of a specialized 
Rogers generating function in terms of the 
continuous $q$-Hermite polynomials 
defined as
\[
H_n(x|q):=e^{in\theta}\qhyp20{q^{n},0}{-}
{q,q^ne^{-2i\theta}},
\]
where $x=\cos\theta$. Using 
\cite[(14.10.34)]{Koekoeketal}
\[
\lim_{\beta\to 0} C_n(x;\beta|q)=
\frac{H_n(x|q)}{(q;q)_n},
\]
one obtains
\begin{equation}
\frac{(t\beta e^{i\theta},t\beta 
e^{-i\theta};q)_\infty}{(te^{i\theta},
te^{-i\theta};q)_\infty}=\sum_{n=0}^\infty 
\frac{(\beta;q)_n}{(q;q)_n}t^n
\qhyp11{\beta q^{n}}{0}{\beta t^2}
\!H_n(x|q).
\label{4:26}
\end{equation}
One can see that by setting $\beta=0$ 
in (\ref{4:26}) that this is 
a generalization of the generating 
function for continuous $q$-Hermite 
polynomials, namely
\cite[(14.26.11)]{Koekoeketal}
\begin{equation}
\frac{1}{(te^{i\theta},te^{-i\theta}
;q)_\infty}=\sum_{n=0}^\infty 
\frac{t^n}{(q;q)_n} H_n(x|q).
\label{4:27}
\end{equation}

We can also derive an expansion of the 
Rogers generating function in terms of 
the Chebyshev polynomials of the first 
kind $T_n(\cos \theta)=\cos(n\theta)$.
This result follows using 
\cite[p. 474]{Koekoeketal}
\[
\lim_{\beta\to 0}\frac{(q^{\beta+1}
;q)_n}{(q^\beta;q)_n}C_n(x;q^\beta|q)
=\epsilon_n\,T_n(x),
\]
where $\epsilon_n:=2-\delta_{n,0}$, is 
called the Neumann factor.

\begin{cor}
Let $x\in[-1,1]$, $|t|<1$, $\beta,\gamma
\in (-1,1)\setminus\{0\}$, $0<|q|<1$.
Then
\begin{equation}
\frac{(t\beta e^{i\theta},t\beta e^{-i\theta};q)_\infty}
{(te^{i\theta},te^{-i\theta};q)_\infty}=\sum_{n=0}^\infty 
\epsilon_n 
\frac{(\beta;q)_n}{(q;q)_n}
t^n 
\qhyp21{\beta,\beta q^{n}}{q^{n+1}}{t^2}
\!
T_n(x).
\label{4:28}
\end{equation}
\end{cor}
\noindent The above formula is a $q$-analogue of 
\cite[(3.10)]{CohlDominici}
\begin{equation}
\frac{1}{(z-x)^\nu}=
\sqrt{\frac{2}{\pi}}
\frac{
e^{i\pi(\frac12-\nu)}
}
{
\Gamma(\nu)
(z^2-1)^{\frac{\nu}{2} -\frac14}
}
\sum_{n=0}^\infty\epsilon_n Q_{n-\frac12}^{\nu
-\frac12}(z)T_n(x). \label{4:29}
\end{equation}
The above result (\ref{4:28})
can be specialized to find a $q$-analogue of Heine's 
reciprocal square root identity \cite[p. 286]{Heine}
\begin{equation}
\frac{1}{\sqrt{z-x}}=\frac{\sqrt{2}}{\pi}
\sum_{n=0}^\infty \epsilon_n Q_{n-1/2}(z) T_n(x),
\label{4:30}
\end{equation}
namely (\ref{4:28}) with $\beta=q^{1/2}$.
Note that the Legendre function of the second 
kind $Q_\nu^\mu:{\mathbb C}\setminus(-\infty,1]
\to{\mathbb C}$ can be defined in terms of the 
Gauss hypergeometric function
\[
Q_\nu^\mu(z):=\frac{\sqrt{\pi}\,e^{i\pi\mu}
\Gamma(\nu+\mu+1)(z^2-1)^{\mu/2}}{2^{\nu+1}
\Gamma(\nu+3/2)z^{\nu+\mu+1}}
\hyp21{\frac{\nu+\mu+1}{2},\frac{\nu+\mu+2}{2}}
{\nu+\frac 32}{\frac{1}{z^2}},
\]
for $|z|>1$ and $\nu+\mu+1\notin-{\mathbb N}_0$ 
(cf. Section 14.21 and (14.3.7) in 
\cite{NIST:DLMF}), and we have used the 
common convention $Q_\nu:=Q_\nu^0$.

Furthermore, (\ref{4:21}) produces the 
following result in terms of the continuous
$q$-Legendre polynomials defined in terms of 
the continuous 
$q$-ultraspherical/Rogers polynomials by 
\cite[p. 478]{Koekoeketal}
\[
P_n(x|q):=q^{n/4}C_n(x;q^{1/2}|q).
\]
\begin{cor}
Let $x\in[-1,1]$, $|t|<1$, $\beta,\gamma\in 
(-1,1)\setminus\{0\}$, $0<|q|<1$.
Then
\begin{equation}
\frac{(t\beta e^{i\theta},t\beta e^{-i\theta};q)_\infty}
{(te^{i\theta},te^{-i\theta};q)_\infty}=\sum_{n=0}^\infty 
\frac{(\beta;q)_n}{(q^{1/2};q)_n}
(tq^{-1/4})^n 
\qhyp21{\beta q^{-1/2},\beta q^{n}}{q^{n+3/2}}{q^{1/2}t^2}
\!P_n(x|q).
\label{genHeinesform}
\end{equation}
\end{cor}

\noindent 
The relevant limiting procedure is \cite[(14.10.49)]{Koekoeketal} 
\begin{equation}
\lim_{q\uparrow1^{-}} P_n(x|q)=P_n(x),
\label{limitingprocedLegenq}
\end{equation}
where $P_n$ is the Legendre polynomial defined by \cite[(9.8.62)]{Koekoeketal}
\[
P_n(x):=
\hyp21{-n,n+1}{1}{\frac{1-x}{2}}.
\]
Using (\ref{limitingprocedLegenq}), one can see
that (\ref{genHeinesform}) is a $q$-analogue of \cite[(14)]{CohlGenGegen}, namely
\begin{equation}
\frac{1}{(z-x)^\nu}=
\frac{e^{i\pi(1-\nu)}(z^2-1)^{(1-\nu)/2}}{\Gamma(\nu)}
\sum_{n=0}^\infty (2n+1) Q_n^{\nu-1}(z) P_n(x),
\label{genGegLeg}
\end{equation}
which is a generalization of Heine's formula \cite{Heine1878} 
\begin{equation}
\frac{1}{z-x}=\sum_{n=0}^\infty (2n+1) Q_n(z) P_n(x).
\label{Heinesformula}
\end{equation}
The $q$-analogue of Heine's formula is
(\ref{genHeinesform}) with $\beta=q$.

\medskip
The above analysis is summarized as a hierarchical scheme 
in Figures \ref{figure1} and \ref{figure2}.

\begin{sidewaysfigure}
\caption{
A heirarchy of generalized generating functions which 
connects expansions of classical and $q$-hypergeometric 
orthogonal polynomials for the continuous $q$-Legendre, 
Legendre, continuous $q$-ultraspherical/Rogers, Gegenbauer, 
Chebyshev of the first kind, and continuous $q$-Hermite 
polynomials.
}
\begin{tikzpicture}[level distance=.05cm,sibling distance=.04cm,scale=0.578,every node/.style={scale=0.62}
,inner sep=11pt]

\node (c) at (12.725,-5.71) {
\fbox{
$\begin{array}{cc}
{\displaystyle \,}\\[-10pt]
{\displaystyle \hspace{-0.1cm}\frac{1}{(z-x)^\nu}=
\frac{(z^2-1)^{\frac{\nu-1}{2}}
}{\Gamma(\nu)
e^{i\pi(\nu-1)}}}
{\displaystyle  \sum_{n=0}^\infty (2n+1) Q_n^{\nu-1}(z) 
P_n(x)\!\!\!\!\!\!} \\[14pt]  \hspace{-0cm}\text{
(\ref{genGegLeg}) : (13) in Cohl (2013) 
\cite{CohlGenGegen}}
\\[6pt]
{\displaystyle \frac{(t\beta e^{i\theta},t\beta 
e^{-i\theta};q)_\infty}{(te^{i\theta},te^{-i\theta};
q)_\infty}=\sum_{n=0}^\infty \frac{(\beta;q)_n}
{(q^\frac12;q)_n}(tq^{-\frac14})^n \qhyp21{\beta 
q^{-\frac12},\beta q^{n}}{q^{n+\frac32}}{q^\frac 
12t^2}\!P_n(x|q)}\\[18pt]
\text{
(\ref{genHeinesform}) :
$q$-analogue (continuous $q$-Legendre polynomials)}
\end{array}$
}};

\node (a) at (22.1,-12.7) {
{\setlength{\fboxrule}{.055cm}
\fbox{
$\begin{array}{cc}
{\displaystyle \,}\\[-10pt]
{\displaystyle \frac{1}{z-x}=\sum_{n=0}^\infty (2n+1)
Q_n(z) P_n(x)}\\[15pt]\text{
(\ref{Heinesformula}) :
Heine (1878) \cite{Heine1878} Heine's formula 
}\\[10pt]
{\displaystyle \,}\\[-17pt]
{\displaystyle \frac{(tqe^{i\theta},tqe^{-i\theta};
q)_\infty}{(te^{i\theta},te^{-i\theta};q)_\infty}
=\sum_{n=0}^\infty \frac{(q;q)_n}{(q^\frac12;q)_n}
(tq^{-\frac14})^n 
\qhyp21{q^\frac12,q^{n+1}}{q^{n+\frac32}}{q^\frac12t^2}
\!P_n(x|q)}\\[15pt]
\text{
(\ref{genHeinesform}) with $\beta=q$ :
$q$-analogue (continuous $q$-Legendre polynomial)}
\end{array}$
}
}
}; 

\node (ap) at (21.9,7.00) {
\fbox{
$\begin{array}{cc}
{\displaystyle \,}\\[-10pt]
{\displaystyle \frac{1}{z-x}=\frac{2^{\mu+\frac12}
\Gamma(\mu)e^{i\pi(\mu-\frac12)}}{\sqrt{\pi}\,(z^2-1)^{-\frac{\mu}{2}+\frac14}}\sum_{n=0}^\infty 
(n+\mu) Q_{n+\mu-\frac12}^{-\mu+\frac12}(z) C_n^\mu(x)}
\\[12pt]\text{(7.2) in Durand 
et al. (1976) \cite{DurandFishSim}}
{\displaystyle \,}\\[10pt]
{\displaystyle \frac{(tqe^{i\theta},tqe^{-i
\theta};q)_\infty}{(te^{i\theta},te^{-i\theta};
q)_\infty}=\sum_{n=0}^\infty\frac{(q;q)_n}
{(\gamma;q)_n}t^n\qhyp21{q\gamma^{-1},q^{n+1}}
{\gamma q^{n+1}}{\gamma t^2}\!C_n(x;\gamma|q)}\\[15pt]
\text{
(\ref{4:21}) with $\beta=q$ : $q$-analogue 
(continuous $q$-ultraspherical/Rogers polynomials)}
\end{array}$
}
};


\node (ctshg) at (15.145,-1.20) {
{\setlength{\fboxrule}{.056cm}
\fbox{
$\begin{array}{cc}
{\displaystyle \,}\\[-10pt]
{\displaystyle \frac{1}{(te^{i\theta},te^{-i\theta}
;q)_\infty}=\sum_{n=0}^\infty \frac{t^n}{(q;q)_n}
H_n(x|q)}\\[15pt]
\text{
(\ref{4:27}) :
generating function for continuous $q$-Hermite polynomials}
&\end{array}$
}
}
};

\node (ctsh) at (11.645,2.50) {
\fbox{
$\begin{array}{cc}
{\displaystyle \,}\\[-10pt]
{\displaystyle \frac{(t\beta e^{i\theta},t\beta e^{-i\theta};q)_\infty}
{(te^{i\theta},te^{-i\theta};q)_\infty}=\sum_{n=0}^\infty 
\frac{(\beta;q)_n}{(q;q)_n}t^n
\qhyp11{\beta q^{n}}{0}{\beta t^2}
\!H_n(x|q)
}\\[15pt]
\text{
(\ref{4:26}) :
$q$-expansion (continuous $q$-Hermite polynomials)}
&\end{array}$
}
};


\node (e) at (-3.935,-12.7) {
{\setlength{\fboxrule}{.055cm}
\fbox{
$\begin{array}{cc}
{\displaystyle \,}\\[-10pt]
{\displaystyle \frac{1}{\sqrt{z-x}}=\frac{\sqrt{2}}{\pi}\sum_{n=0}^\infty \epsilon_n
Q_{n-\frac12}(z) T_n(x)}
\\[15pt]\text{
(\ref{4:30}) : Heine (1881) \cite{Heine} reciprocal square root identity (1881)}\\[10pt] {\displaystyle \frac{(tq^\frac12e^{i\theta},tq^\frac12e^{-i\theta};q)_\infty}
{(te^{i\theta},te^{-i\theta};q)_\infty}=\sum_{n=0}^\infty 
\epsilon_n
\frac{(q^\frac12;q)_n}{(q;q)_n}
t^n 
\qhyp21{q^\frac12,q^{n+\frac12}}{q^{n+1}}{t^2}
\!T_n(x)
}\\[15pt]
\text{
(\ref{4:28}) with $\beta=q^{\frac12}$ : 
$q$-analogue (Chebyshev polynomials of the first kind)}
\end{array}$
}
}
};

\node (ep) at (-4.355,-6.71) {
\fbox{
$\begin{array}{cc}
{\displaystyle \,}\\[-10pt]
{\displaystyle 
\frac{1}{(z-x)^\nu}=
\sqrt{\frac{2}{\pi}}\frac{(z^2-1)^{-\frac{\nu}{2}+\frac14}}{e^{i\pi(\nu-\frac12)}\Gamma(\nu)}
\sum_{n=0}^\infty\epsilon_n 
Q_{n-\frac12}^{\nu-\frac12}(z)T_n(x)}\\[13pt]
\text{(\ref{4:29}) : 
(3.10) in Cohl \& Dominici (2011) \cite{CohlDominici}}\\[4pt]
{\displaystyle \,}\\[-10pt]
{\displaystyle \frac{(t\beta e^{i\theta},t\beta e^{-i\theta};q)_\infty}
{(te^{i\theta},te^{-i\theta};q)_\infty}=\sum_{n=0}^\infty 
\epsilon_n 
\frac{(\beta;q)_n}{(q;q)_n}
t^n 
\qhyp21{\beta,\beta q^{n}}{q^{n+1}}{t^2}
\!
T_n(x)
}\\[17pt]
\text{
(\ref{4:28}) :
$q$-analogue (Chebyshev polynomial of the first kind)}
\end{array}$
}
};

\node (b) at (8.645,-12.7) {
{\setlength{\fboxrule}{.055cm}
\fbox{
$\begin{array}{cc}
{\displaystyle \,}\\[-10pt]
{\displaystyle \frac{1}{(1+t^2-2t x)^\nu}=\sum_{n=0}^\infty t^n C_n^\nu(x)}
\\[15pt]\text{
(\ref{4:23}) : 
Gegenbauer (1874) \cite{Gegenbauer1874} generating function}\\[5pt]
{\displaystyle \,}\\[-10pt]
{\displaystyle
\frac{(t\beta e^{i\theta},t\beta e^{-i\theta};q)_\infty}
{(t e^{i\theta},t e^{-i\theta};q)_\infty}
=\sum_{n=0}^\infty t^n C_n(x;\beta|q) 
}
\\[15pt]\text{ (\ref{4:22}) :
Rogers (1893) \cite{rogers} generating function 
}\\[4pt]\end{array}$
}
}
};

\node (bp) at (-4.235,7.00) {
\fbox{
$\begin{array}{cc}
{\displaystyle \,}\\[-10pt]
{\displaystyle \frac{1}{(z-x)^\nu}=
\frac{2^{\mu+\frac12}\Gamma(\mu)
e^{i\pi(\mu-\nu+\frac12)}}
{\sqrt{\pi}\,\Gamma(\nu)(z^2-1)^{\frac
{\nu-\mu}{2}-\frac14}}}
{\displaystyle 
\sum_{n=0}^\infty (n+\mu) Q_{n+\mu
-\frac12}^{\nu-\mu-\frac12}(z)C_n^\mu(x)}
\\[15pt]\text{
(\ref{4:25}) :
Theorem 2.1 in Cohl (2013) \cite{CohlGenGegen}
}\\[10pt]
{\displaystyle
\frac{(t\beta e^{i\theta},t\beta e^{-i\theta};q)_\infty}
{(t e^{i\theta},t e^{-i\theta};q)_\infty}
=\sum_{n=0}^\infty
\frac{(\beta;q)_n}{(\gamma;q)_n}
t^n
\,{}_2\phi_1\left(
\begin{array}{c}
\beta\gamma^{-1},\beta q^n\\ \gamma q^{n+1}
\end{array} ;q,\gamma t^2 \right)
\!C_n(x;\gamma|q)
}
\\[18pt]\text{ 
(\ref{4:21}) :
$q$-analogue (continuous $q$-ultraspherical/Rogers polynomials)
}
\end{array}$
}
};

\draw[<-,line width=1.5pt] (ctshg) -- (ctsh);  
\draw[<-,line width=1.5pt] (ctsh) -- (bp);
\draw[<-,line width=1.5pt] (b) -- (bp);
\draw[<-,line width=1.5pt] (a) -- (c);
\draw[<-,line width=1.5pt] (a) -- (ap); 
\draw[<-,line width=1.5pt] (c) -- (bp);  
\draw[<-,line width=1.5pt] (ap) -- (bp);
\draw[<-,line width=1.5pt] (e) -- (ep);   
\draw[<-,line width=1.5pt] (ep) -- (bp);

\end{tikzpicture}
\label{figure1}
\end{sidewaysfigure}
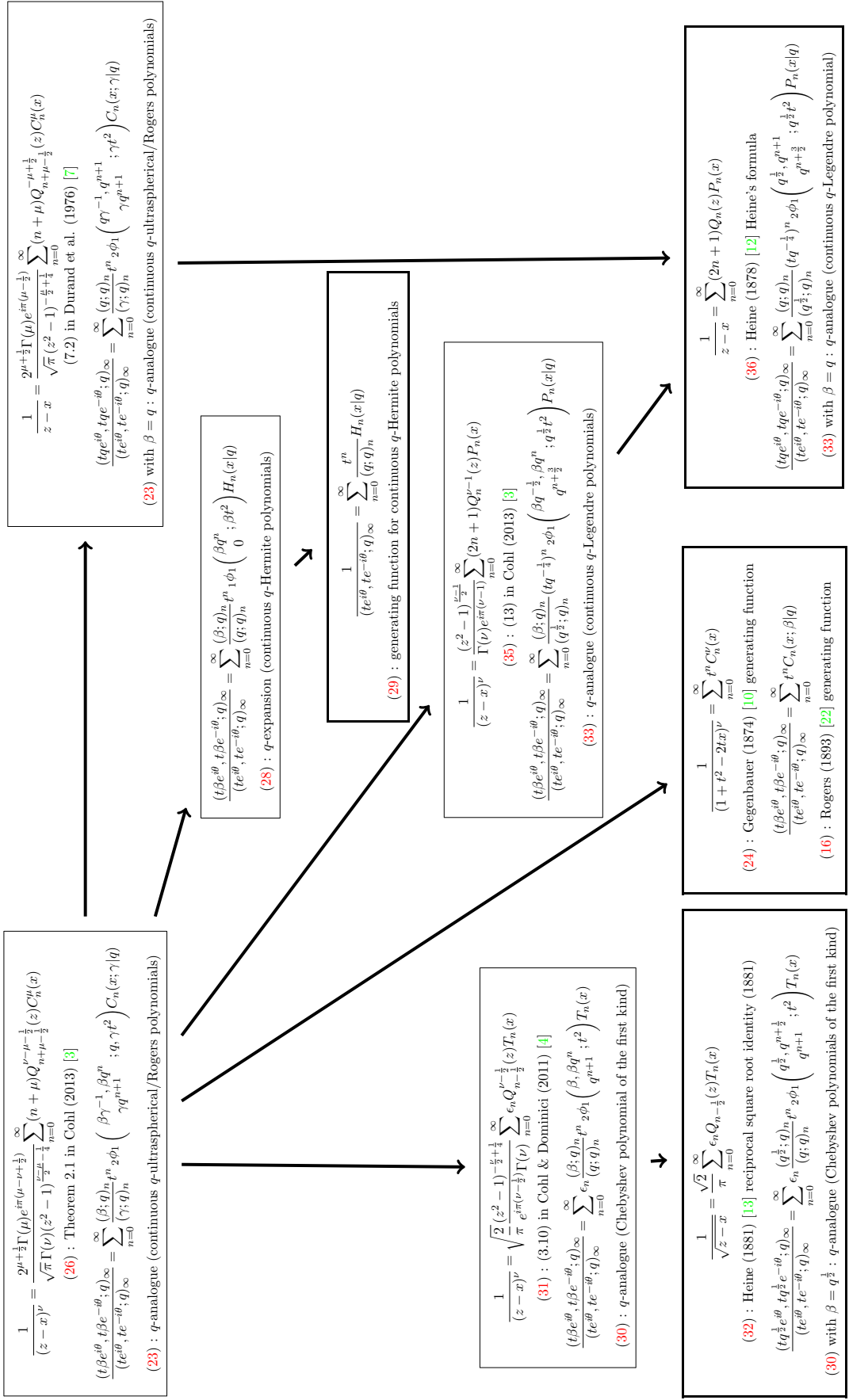

\begin{figure}[!htb]
\caption{A heirarchy of generalized generating functions 
for the continuous $q$-ultraspherical/Rogers polynomials 
which connects expansions of classical and 
$q$-hypergeometric orthogonal polynomials for the 
continuous $q$-ultraspherical/Rogers, Gegenbauer, 
continuous $q$-Jacobi, Jacobi, and Askey-Wilson polynomials.}
\begin{tikzpicture}[level distance=.05cm,sibling distance=.04cm,scale=0.62,every node/.style={scale=0.62},inner sep=11pt]

\node (bp) at (10.70,-34.80) {
{\setlength{\fboxrule}{.055cm}
\fbox{
$\begin{array}{cc}
{\displaystyle \,}\\[-10pt]
{\displaystyle \frac{1}{(z-x)^\nu}=
\frac{2^{\mu+\frac12}\Gamma(\mu)e^{i\pi(\mu-\nu+\frac12)}}
{\sqrt{\pi}\,\Gamma(\nu)(z^2-1)^{\frac{\nu-\mu}{2}-\frac14}}}
{\displaystyle 
\sum_{n=0}^\infty
(n+\mu)
Q_{n+\mu-\frac12}^{\nu-\mu-\frac12}(z)C_n^\mu(x)}
\\[15pt]\text{
(\ref{4:25}) :
Theorem 2.1 in Cohl (2013) \cite{CohlGenGegen}
}\\[10pt]
{\displaystyle
\frac{(t\beta e^{i\theta},t\beta e^{-i\theta};q)_\infty}
{(t e^{i\theta},t e^{-i\theta};q)_\infty}
=\sum_{n=0}^\infty
\frac{(\beta;q)_n}{(\gamma;q)_n}
t^n
\,{}_2\phi_1\left(
\begin{array}{c}
\beta\gamma^{-1},\beta q^n\\
\gamma q^{n+1}
\end{array}
;q,\gamma t^2
\right)
\!C_n(x;\gamma|q)
}
\\[18pt]\text{ (\ref{4:21}) :
$q$-analogue (continuous $q$-ultraspherical/Rogers polynomials)
}
\end{array}$
}
}
};

\node (app) at (0.24,-29.40) {
\fbox{
$\begin{array}{cc}
{\displaystyle \,}\\[-10pt]
{\displaystyle \frac{1}{z-x}=
\frac{(z-1)^\alpha(z+1)^\beta}{2^{\alpha+\beta}}}
{\displaystyle\sum_{n=0}^\infty 
\frac{(2n+\alpha+\beta+1)\Gamma(\alpha+\beta+n+1)n!}
{\Gamma(\alpha+1+n)\Gamma(\beta+1+n)}
Q_n^{(\alpha,\beta)}(z) 
P_n^{(\alpha,\beta)}(x)}\hspace{-0.3cm}\\[15pt]
\hspace{-0cm}\text{(9.2.1) in Szeg\H{o} (1959)
\cite{Szego}
}\\[10pt]
{\displaystyle 
\frac{(t\beta e^{i\theta},t\beta e^{-i\theta};q)_\infty}
{(t e^{i\theta},t e^{-i\theta};q)_\infty}
=
\sum_{n=0}^\infty
d_n(q,t,\alpha,\gamma,q)
P_n^{(\alpha,\gamma)}(x|q)
}
\\[18pt]
\text{
(\ref{3:18}) with 
$\beta=q$ : $q$-analogue (continuous $q$-Jacobi polynomials)
}
&\end{array}$
}
};

\node (appp) at (6.40,-20.20) {
\fbox{
$\begin{array}{cc}
{\displaystyle \,}\\[-10pt]
{\displaystyle \frac{1}{(z-x)^\nu}=
\frac{(z-1)^{\alpha+1-\nu}(z+1)^{\beta+1-\nu}}
{2^{\alpha+\beta+1-\nu}}}
{\displaystyle \sum_{n=0}^\infty 
\frac{(2n+\alpha+\beta+1)\Gamma(\alpha+\beta+n+1)(\nu)_n}
{\Gamma(\alpha+n+1)\Gamma(\beta+n+1)}
Q_{n+\nu-1}^{(\alpha+1-\nu,\beta+1-\nu)}(z) 
P_n^{(\alpha,\beta)}(x)}\\[15pt]
\hspace{-0cm}\text{
Theorem 1 in Cohl (2013) \cite{Cohl12pow}
}\\[9pt]
{\displaystyle 
\frac{(t\beta e^{i\theta},t\beta e^{-i\theta};q)_\infty}
{(t e^{i\theta},t e^{-i\theta};q)_\infty}
=
\sum_{n=0}^\infty
d_n(\beta,t,\alpha,\gamma,q)
P_n^{(\alpha,\gamma)}(x|q)
}\\[20pt]
{\displaystyle 
d_n(\beta,t,\alpha,\gamma,q):=\frac
{(t\alpha^{-\frac12})^n(\beta,
-(\alpha\gamma)^\frac12,-(q\alpha\gamma)^\frac12;q)_n
(q^n\alpha^\frac12\beta t,-q^n\gamma^\frac12\beta t,-q^{n+\frac12}\gamma^\frac12\beta t,
q^{n+\frac12}\alpha^\frac12\gamma t;q)_\infty
}
{
(q^n\alpha\gamma;q)_n
(\alpha^\frac12t,-\gamma^\frac12t,-(q\gamma)^\frac12t,
q^{2n+\frac12}\alpha^\frac12\gamma\beta t;q)_\infty}
}\\[20pt]
{\displaystyle
\times
\qhyp87{
q^{2n-\frac12}\alpha^\frac12\gamma\beta t,
\pm q^{n+\frac34}\alpha^\frac14(\gamma\beta t)^\frac12,
-q^n(\alpha\gamma)^\frac12,
-q^{n+\frac12}(\alpha\gamma)^\frac12,
q^{n+\frac12}\gamma,
(q\alpha)^{-\frac12}\beta t,
q^n\beta
}{
\pm q^{n-\frac14}\alpha^\frac14(\gamma\beta t)^\frac12,
q^n\alpha^\frac12\beta t,
-q^n\gamma^\frac12\beta t,
-q^{n+\frac12}\gamma^\frac12\beta t,
q^{2n+1}\alpha\gamma,
q^{n+\frac12}\alpha^\frac12\gamma t
}{q,(q\alpha)^\frac12t}
}
\\[18pt]
\text{
(\ref{3:18}) : $q$-analogue 
(continuous $q$-Jacobi polynomials)
}
&\end{array}
$ 
}
};

\node (apppp) at (6.00,-8.50) {
\fbox{
$\begin{array}{cc}
{\displaystyle \,}&\\[-10pt]
{\displaystyle 
\frac{\Gamma(t+ix)\Gamma(t-ix)}{\Gamma(t\beta+ix)\Gamma(t\beta-ix)}
=\sum_{n=0}^\infty 
\frac{(t(\beta-1))_n
(a_1+a_2+a_3+n+t)_{n+t(\beta-1)}
W_n(x^2;a_1,a_2,a_3,a_4) 
}
{n!(n-1+a_1+a_2+a_3+a_4)_n
(a_1+t,a_2+t,a_3+t)_{n+t(\beta-1)}}
}&\\[20pt]
{\displaystyle 
\times {}_7F_6\left(\begin{array}{c} 
2n-1+a_1+a_2+a_3+t\beta, \frac{2n+1+a_1+a_2+a_3+t\beta}{2}, n+a_1+a_2, 
n+a_1+a_3, n+a_2+a_3, t\beta-a_4, n+t(\beta-1) \\[1pt]
\frac{2n-1+a_1+a_2+a_3+t\beta}{2}, n+a_1+t\beta, n+a_2+t\beta,
n+a_3+t\beta, 2n+a_1+a_2+a_3+a_4, n+a_1+a_2+a_3+t \nonumber
\end{array};1\right) 
}&\\[19pt]
\text{
(\ref{Wilsonlimit}) : Wilson limit of Ismail-Simeonov generalized generating function}&\\[10pt]
{\displaystyle 
\frac{(t\beta e^{i\theta},t\beta e^{-i\theta};q)_\infty}
{(t e^{i\theta},t e^{-i\theta};q)_\infty}
=
\sum_{n=0}^\infty
c_n(\beta,t,a_1,a_2,a_3,a_4,q)
p_n(x;a_1,a_2,a_3,a_4|q)
}\\[20pt]

{\displaystyle 
c_n(\beta,t,a_1,a_2,a_3,a_4,q):=
\frac
{t^n(\beta;q)_n(q^na_1u,q^na_2u,q^na_3u,tq^na_1a_2a_3;q)_\infty}
{(q,q^{n-1}a_1a_2a_3a_4;q)_n(ta_1,ta_2,ta_3,q^{2n}a_1a_2a_3\beta t;q)_\infty}}\\[20pt]

{\displaystyle
\times
{\displaystyle \qhyp87
{q^{2n-1}a_1a_2a_3\beta t,
\pm q^{n+\frac12}(a_1a_2a_3\beta t)^\frac12,
q^na_1a_2,q^na_1a_3,
q^na_2a_3,\beta t/a_4,q^n\beta}
{
\pm q^{n-\frac12}(a_1a_2a_3\beta t)^\frac12,
q^na_1\beta t,q^na_2\beta t,q^na_3\beta t,q^{2n}a_1a_2a_3a_4,
tq^na_1a_2a_3}{q,ta_4}}

}\\[19pt]
\text{
(\ref{3:15}) : expansion of Rogers 
generating function in Askey-Wilson polynomials
}
&\end{array}
$ 
}
};

\draw[<-,line width=1.5pt] (bp) -- (appp);
\draw[<-,line width=1.5pt] (app) -- (appp);
\draw[<-,line width=1.5pt] (appp) -- (apppp);

\end{tikzpicture}
\label{figure2}
\end{figure}

\subsection{A quadratic transformation for basic hypergeometric functions}

In (\ref{3:15}), let 
$a_1\mapsto \gamma^\frac12$, 
$a_2\mapsto -\gamma^\frac12$, 
$a_3\mapsto -(q\gamma)^\frac12$, 
$a_4\mapsto (q\gamma)^\frac12$, 
and specializing the Askey-Wilson polynomials
to continuous $q$-ultraspherical/Rogers polynomials using
\cite[p. 472]{Koekoeketal}
\[
C_n(x;\gamma|q)=\frac{(\gamma^2;q)_n}{\left(q,-\gamma, 
\pm{q}^\frac12 \gamma;q\right)_n}
p_n\left(x;\gamma^\frac12,
-\gamma^\frac12,
-(q\gamma)^\frac12,
(q\gamma)^\frac12|q\right),
\]
produces an expansion of the Rogers generating function whose coefficients are an 
${}_8\phi_7$.  By comparing the coefficients of this expansion with the 
generalized Rogers generating function (\ref{4:21}), and further replacing 
$(\beta,\gamma)\mapsto (q^{-n}\beta,q^{-n}\gamma)$, $t\mapsto (q/\gamma)^\frac12t$, 
we derive
{\small \begin{eqnarray}
&&\qhyp21{\beta/\gamma,\beta}{q\gamma}{q,qt^2}=
\frac{(q(\beta t)^2,q\gamma t,qt;q)_\infty}
{(q\beta\gamma t,q\beta t,qt^2;q)_\infty}
\qhyp87
{\beta\gamma t,\pm(q^2\beta\gamma t)^\frac12,\pm q^\frac12\gamma,
\beta\gamma^{-1}t,-\gamma,\beta}
{
\pm (\beta\gamma t)^\frac12,\pm q^{\frac12}\beta t,q\gamma^2,-q\beta t,q\gamma t
}{q,qt}.
\label{getCor44IsmailSim}
\end{eqnarray}}
This is a generalization of \cite[Corollary 4.4]{IsmailSimeonov2015} with $\beta=\gamma$.
By re-expressing (\ref{getCor44IsmailSim}), we see that our procedure has produced 
a new {\it quadratic transformation}
for basic hypergeometric functions (see \cite{RahmanVerma93}).
\begin{thm}
Let $t\in(-1,1)$, $0<|q|<1$. Then
{\small \begin{eqnarray*}
&&\qhyp21{a,b}{qab^{-1}}{q,qt^2}
=\frac
{\left(q(at)^2,qab^{-1}t,qt;q\right)_\infty}
{\left(qa^2b^{-1}t,qat,qt^2;q\right)_\infty}
\qhyp87{a^2b^{-1}t,\pm qab^{-\frac12}
t^\frac12,\pm q^{\frac12}ab^{-1},-ab^{-1},
bt,a
}
{\pm ab^{-\frac12}t^\frac12,\pm q^{\frac12}at,qa^2b^{-2},qab^{-1}t,-qat}
{q,qt},
\end{eqnarray*}}
which is valid under the transformation $t\mapsto -t$.
\end{thm}

This quadratic transformation has some interesting consequences. 
For $a=0$ one obtains the $q$-binomial theorem (\ref{2:10}).
For $t\in\mathbb C$, $t= iqb^{-1/2}$, the ${}_2\phi_1$ can be summed 
by the $q$-Kummer (Bailey-Daum) summation
\cite[(II.9)]{GaspRah} (this leads to a very unusual summation 
of the ${}_8\phi_7$).
It corresponds in the $q\uparrow 1^{-}$ limit to 
the quadratic transformation for the Gauss hypergeometric function
\cite[(15.8.21), (15.8.1)]{NIST:DLMF} 
\begin{equation}
\hyp21{a,b}{a-b+1}{t^2}=\frac{1}{(1\pm t)^{2a}}
\hyp21{a,a-b+\frac12}{2a-2b+1}{\frac{\pm 4t}{(1\pm t)^2}}.
\label{quadGausshypergeom}
\end{equation}
Note that \cite[(4.1)]{RahmanVerma93} is a quadratic transformation of basic 
hypergeometric series which in the limit $q\uparrow 1^{-}$ yields (\ref{quadGausshypergeom}). 
However, this quadratic transformation seems altogether different.

\subsection{Expansions of $(1-x)^{-\nu}$}

From the Jacobi expansion of $(z-x)^{-\nu}$ (\ref{3:19}), we can derive an expansion of
$(1-x)^{-\nu}$ by using the limit as $z\to 1^{+}$. Also, this is the corresponding limit
of the Wilson polynomial expansion (\ref{Wilsonlimit}) to the Jacobi polynomials. 
In this subsection we derive this and other limiting expansions, which generalize 
\cite[(18.18.15)]{NIST:DLMF} for $\nu=-n$, $n\in\N_0$.

\begin{cor} Let $x\in(-1,1)$, $\nu\in\C$, $\alpha,\beta\in\C$ 
such that $\Re(\alpha-\nu+1)>0$. Then
\begin{equation}
\frac{1}{(1-x)^\nu}=\frac{\Gamma(\alpha-\nu+1)}{2^\nu}
\sum_{n=0}^\infty
\frac{(\alpha+\beta+2n+1)\Gamma(\alpha+\beta+1+n)(\nu)_n}
{\Gamma(\alpha+1+n)\Gamma(\alpha+\beta+2-\nu+n)}
P_n^{(\alpha,\beta)}(x).
\label{Jacobi1mnux}
\end{equation}
\end{cor}
\begin{proof}[Proof 1]
Start with (\ref{3:19}) and examine the singular behavior of the Jacobi
function of the second kind $Q_\gamma^{(\alpha,\beta)}(z)$ as $z\to 1^{+}$.
Starting with the definition in terms of the Gauss hypergeometric function
of the Jacobi function of the second kind, and applying \cite[(15.8.2)]{NIST:DLMF},
results in the identity
\begin{eqnarray}
&&Q_\gamma^{(\alpha,\beta)}(z)=-\frac{\pi}{2}\csc(\pi\alpha)P_\gamma^{(\alpha,\beta)}(z)\nonumber\\
&&\hspace{1cm}+\frac{2^{\alpha+\beta-1}\Gamma(\alpha)\Gamma(\beta+\gamma+1)}
{\Gamma(\alpha+\beta+\gamma+1)(z-1)^\alpha(z+1)^{\beta}}
\hyp21{\gamma+1,-\alpha-\beta-\gamma}{1-\alpha}{\frac{1-z}{2}},
\label{sing1QJac}
\end{eqnarray}
where $P_\gamma^{(\alpha,\beta)}(z)$ is the Jacobi function of the first kind defined by
\[
P_\gamma^{(\alpha,\beta)}(z)=\frac{\Gamma(\alpha+\gamma+1)}{\Gamma(\alpha+1)\Gamma(\gamma+1)}
\hyp21{-\gamma,\alpha+\beta+\gamma+1}{\alpha+1}{\frac{1-z}{2}}.
\]
Note that $P_\gamma^{(\alpha,\beta)}(z)$ generalizes the Jacobi polynomials for $\gamma=n\in\N_0$.
Using (\ref{sing1QJac}), easily demonstrates that 
\[
(z-1)^{\alpha-\nu+1}Q_{n+\nu-1}^{(\alpha+1-\nu,\beta+1-\nu)}(z)
\sim\frac{2^{\alpha-\nu}\Gamma(\alpha+1-\nu)\Gamma(\beta+1-\nu)}{\Gamma(\alpha+\beta-\nu+2+n)},
\]
for $\Re(\alpha+1-\nu)>0$, and 
(\ref{Jacobi1mnux}) follows.
\end{proof}

\begin{lemma}
Let $a,b\in\C$. Then we have as $0<\tau\to\infty$,
\begin{equation}
\frac{\Gamma(a\pm i\tau)}{\Gamma(b\pm i\tau)}=e^{\frac{i\pi}{2}(a-b)}\tau^{a-b}
\left\{1+\mathcal{O}(\tau^{-1})\right\}.
\label{gammaimaglimits}
\end{equation}
\end{lemma}
\begin{proof}
Let $\delta\in(0,\pi)$. From \cite[(5.11.13)]{NIST:DLMF}, as $z\to\infty$ with $a$ and $b$
real or complex constants, provided $\arg z\le \pi-\delta (<\pi)$.
If one takes $z=\pm i\tau$ with $\tau>0$ then the argument restriction implies
$\arg(\pm i\tau)=\pm \pi/2$, and the result follows.
\end{proof}

\begin{proof}[Proof 2]
One can obtain Jacobi polynomials from Wilson polynomials using 
\cite[(9.1.18)]{Koekoeketal}
\begin{equation}
P_n^{(\alpha,\beta)}(x)=\lim_{\tau\to\infty}\frac{1}{\tau^{2n}n!}
W_n\left(
\frac{(1-x)t^2}{2};
\frac{\alpha+1}{2},
\frac{\alpha+1}{2},
\frac{\beta+1}{2}+i\tau,
\frac{\beta+1}{2}-i\tau
\right).
\label{JacobiWilsonlimit}
\end{equation}
Define $\nu=u-t$. 
Apply (\ref{JacobiWilsonlimit}) to (\ref{Wilsonlimit}) using (\ref{gammaimaglimits}) repeatedly, one obtains
\begin{eqnarray*}
&&\frac{1}{(1-x)^\nu}=\frac{1}{2^\nu}
\frac{
\left(
\frac{\alpha+1}{2},
\frac{\alpha+1}{2}
\right)_t
}
{
\left(
\frac{\alpha+1}{2},
\frac{\alpha+1}{2}
\right)_{t+\nu}
}
\sum_{n=0}^\infty
\frac{(\nu,\alpha+\beta+1)_n}{
(\frac{\alpha+1}{2}+u,\frac{\alpha+1}{2}+u)_n(\alpha+\beta+1)_{2n}
}P_n^{(\alpha,\beta)}(x)\nonumber\\
&&\hspace{2cm}\times
\lim_{\tau\to\infty} \tau^{2\nu+2n}
\hyp23{\alpha+1+n,\nu+n}{
\frac{\alpha+1}{2}+t+\nu+n,
\frac{\alpha+1}{2}+t+\nu+n,
\alpha+\beta+2+2n}{-\tau^2}.
\end{eqnarray*}
The above limit of the ${}_2F_3$ can be computed using the asymptotic expansion for large
variables of the generalized hypergeometric function 
\cite[(16.11.8)]{NIST:DLMF} assuming $\Re(\alpha+1-\nu)>0$. This completes the proof.
\end{proof}

From (\ref{Jacobi1mnux}), one can derive some interesting specialization and limit consequences.

\begin{cor}
Let $x\in(-1,1)$, $\mu\in(-\frac12,\infty)\setminus\{0\}$, $\nu\in\C$, such that $\Re(\mu-\nu+\tfrac12)>0$. Then
\begin{equation}
\frac{1}{(1-x)^\nu}=\frac{2^{2\mu-\nu}\Gamma(\mu-\nu+\frac12)\Gamma(\mu)}
{\sqrt{\pi}\,\Gamma(2\mu+1-\nu)}
\sum_{n=0}^\infty
\frac{(\mu+n)(\nu)_n}
{(2\mu+1-\nu)_n}
C_n^\mu(x).
\label{Geg1mnux}
\end{equation}
\end{cor}
\begin{proof}
Specializing (\ref{Jacobi1mnux}) using the definition of Gegenbauer polynomials in terms of 
Jacobi polynomials (\ref{GegenJacobi}),
which completes the proof.
\end{proof}

\begin{cor}
Let $x\in(-1,1)$, $\mu\in(-\frac12,\infty)\setminus\{0\}$, $\nu\in\C$, such that $\Re(\mu-\nu+\tfrac12)>0$. Then
\begin{equation}
\frac{1}{(1-x)^\nu}=\frac{\Gamma(\frac12-\nu)}
{\sqrt{\pi}\,2^\nu\Gamma(1-\nu)}
\sum_{n=0}^\infty
\frac{
\epsilon_n
(\nu)_n}
{(1-\nu)_n}
T_n(x),
\label{sum1mncCheby}
\end{equation}
where $\epsilon_n:=2-\delta_{n,0}$ is the Neumann factor.
\end{cor}
\begin{proof}
Specializing (\ref{Geg1mnux}) using the limit relation for Chebyshev polynomials of the first kind $T_n(x)$
with Gegenbauer polynomials, namely \cite[(6.4.13)]{AAR}
\[
\lim_{\mu\to0}\frac{n+\mu}{\mu}C_n^\mu(x)=\epsilon_nT_n(x),
\]
which completes the proof.
\end{proof}

\begin{cor}
Let $x\in(0,\infty)$, $\alpha>-1$, $\nu\in\C$ such that $\Re(\alpha+1-\nu)>0$. 
Then
\begin{equation}
\frac{1}{x^\nu}
=\Gamma(\alpha+1-\nu)
\sum_{n=0}^\infty
\frac{
(\nu)_n}
{\Gamma(\alpha+1+n)}
L_n^\alpha(x).
\label{sum1mncLag}
\end{equation}
\end{cor}
\begin{proof}
Specializing (\ref{Jacobi1mnux}) using the limit relation for Laguerre polynomials $L_n^\alpha(x)$
with Gegenbauer polynomials, namely \cite[(9.8.16)]{Koekoeketal}
\[
\lim_{\beta\to\infty}P_n^{(\alpha,\beta)}\left(1-\frac{2x}{\beta}\right)=L_n^\alpha(x),
\]
which completes the proof.
\end{proof}
The above result generalizes \cite[(18.18.19)]{NIST:DLMF} for $\nu=-n$, $n\in\N_0$.

\section{Definite integrals}
\label{Definiteintegralsinfiniteseriesandqintegrals}

Consider a sequence of orthogonal polynomials 
$(p_k(x;{\boldsymbol \alpha}))$ (over a domain $A$, with 
positive weight $w(x;{\boldsymbol \alpha})$) associated 
with a linear functional ${\bf u}$, where 
${\boldsymbol\alpha}$ is a set of fixed parameters. 
Define $s_k$, $k\in\mathbb N_0$ by
\[
s^2_k:=\int_A 
p_k(x;{\boldsymbol \alpha})
p_k(x;{\boldsymbol \alpha})
\, w(x;{\boldsymbol \alpha})\, dx.
\]
In order to justify interchange between a generalized 
generating function via connection relation and an 
orthogonality relation for $p_k$, we show that the 
double sum/integral converges in the $L^2$-sense with 
respect to the weight $w(x;{\boldsymbol \alpha})$.
This requires
\begin{eqnarray} \label{form-inve-L2}
\sum_{k=0}^\infty d_k^2 s_k^2<\infty,
\end{eqnarray}
where $\displaystyle d_k=\sum_{n=k}^\infty a_n c_{k,n}$.

Here $a_n$ is the coefficient multiplying the 
orthogonal polynomial in the original generating 
function, and $c_{k,n}$ is the connection coefficient 
for $p_k$ (with appropriate set of parameters).

\begin{lemma}
\label{lemmasum}
Let $\bf u$ be a classical linear functional and 
let $(p_n(x))$, $n\in\mathbb N_0$ be the sequence 
of orthogonal polynomials associated with $\bf u$.
If $|p_n(x)|\le K(n+1)^\sigma \gamma^n$, with $K$, 
$\sigma$ and $\gamma$ constants independent of $n$, 
then $|s_n|\le  K(n+1)^\sigma \gamma^n |s_0|$.
\end{lemma}
\begin{proof}
Let $n\in\mathbb N_0$, then
\[
s_n^2=\langle {\bf u}, p^2_n\rangle \le \left(K(n+1)^\sigma 
\gamma^n\right)^2\langle {\bf u}, 1\rangle= \left(K(n
+1)^\sigma \gamma^n\right)^2  s_0^2.
\]
The result follows.
\end{proof}

Given
$|p_k(x;{\boldsymbol \alpha})|\le K(k+1)^\sigma 
\gamma^k$, with $K$, $\sigma$ and $\gamma$ 
constants independent of $k$, an orthogonality 
relation for $p_k$, and $|t|<1/\gamma$, one has
\[
\sum_{n=0}^\infty |a_n| \sum_{k=0}^n |c_{k,n}s_k|<\infty,
\]
which implies
\[
\sum_{k=0}^\infty |d_ks_k|<\infty.
\]
Therefore one has confirmed (\ref{form-inve-L2}),
indicating that we are justified in reversing 
the order of our generalized sums and the 
orthogonality relations under the above assumptions.

All polynomial families used throughout 
this paper fulfill such assumptions.
See for instance (\ref{4:24}).
Such inequalities depend entirely on the 
representation of the linear functional.
In this section one has integral representations, infinite series, and
representations in terms of the $q$-integral. 
In all the cases Lemma \ref{lemmasum} can be 
applied and we are justified in interchanging 
the linear form and the infinite sum.
\medskip

\subsection{The Askey-Wilson polynomials}
 
The orthogonality relation for the Askey-Wilson polynomials is given
by \cite[(14.1.2)]{Koekoeketal}
\begin{equation}
\label{orthoAW}
\int_{-1}^1 p_m(x;{\bf a}|q) p_n(x;{\bf a}|q)\frac{w(x;{\bf a};q)}{\sqrt{1-x^2}}dx
=2\pi h_n({\bf a};q) \delta_{m,n},
\end{equation}
where ${\bf a}:=\{a_1,a_2,a_3,a_4\}$, $w:[-1,1]\to[0,\infty)$ is defined by
\[
w(x;{\bf a};q):=\left|
\frac{(e^{2i\theta};q)_\infty}
{(a_1e^{i\theta},a_2e^{i\theta},a_3e^{i\theta},a_4e^{i\theta};q)_\infty}
\right|^2, \quad x=\cos\theta,
\]
and
\[
h_n({\bf a};q):=
\frac{(a_1a_2a_3a_4q^{n-1};q)_n(a_1a_2a_3a_4q^{2n};q)_\infty}
{(q^{n+1},a_1a_2q^n,a_1a_3q^n,a_1a_4q^n,a_2a_3q^n,a_2a_4q^n,a_3a_4q^n;q)_\infty}.
\]

\begin{cor} Let $n\in\mathbb N_0$, $\beta\in(-1,1)$,
$\max\{|a_1|,|a_2|,|a_3|,|a_4|,|t|\}<1$, 
$c_n(\beta,t,{\bf a};q)$ defined as in 
(\ref{3:15}), $x=\cos\theta\in(-1,1)$.
Then
\[
\int_{-1}^{1}
\frac{(t \beta e^{i\theta},t\beta e^{-i\theta};q)_\infty}
{(te^{i\theta},te^{-i\theta};q)_\infty}\,
p_n(x;{\bf a}\vert q) 
\frac{w(x;{\bf a};q)}{\sqrt{1-x^2}} dx 
= {2\pi}h_n({\bf a};q) c_n(\beta,t,{\bf a};q).
\]
\end{cor}
\begin{proof}
Multiply (\ref{3:15}) by 
$w(x;{\bf a};q) p_n(x;{\bf a}\vert q)/\sqrt{1-x^2}$ 
and integrate over $(-1,1)$ using (\ref{orthoAW})
produces the desired result.
\end{proof}

\subsection{The Wilson polynomials}

\begin{cor}
Let $n\in{\mathbb N}_0$, $t,u\in{\mathbb C}$, $\Re(a_1,a_2,a_3,a_4)>0$, and non-real parameters
occur in conjugate pairs. Then
{\small \[\begin{split}
\int_0^\infty & \dfrac{\Gamma(t+ix)
\Gamma(t-ix)}{\Gamma(u+ix)\Gamma(u-ix)}\, 
W_n(x^2;{\bf a}) {\sf W}(x;{\bf a}) dx=
H_n({\bf a})
\frac
{(a_{123})_u(a_1,a_2,a_3)_t(a_{123}+u)_{2n} (a_{1234}-1)_n }
{(a_{123})_t(a_1,a_2,a_3)_u (a_{123}+t)_n (a_{1234}-1)_{2n}
n!}
\\[3mm]
& \hspace{-3mm}\times\frac{(u-t)_n}
{(a_1+u,a_2+u,a_3+u)_n}\hyp{7}{6}
{\lambda,1+\lambda/2,a_{12}+n,a_{13}+n,
a_{23}+n,u-a_4,u-t+n}{\lambda/2, a_1+u+n,
a_2+u+n,a_3+u+n,a_{123}+t+n,a_{1234}+2n} 1,
\end{split}\]}
where $\lambda=2n-1+a_{123}+u$, the weight 
function for the Wilson polynomials is \cite[(9.1.2)]{Koekoeketal}
\[
{\sf W}(x;{\bf a}):=\left|\dfrac
{\Gamma(a_1+ix)\Gamma(a_2+ix)\Gamma(a_3+ix)
\Gamma(a_4+ix)}{\Gamma(2ix)}\right|^2,
\]
and  \cite[(9.1.2)]{Koekoeketal}
\begin{eqnarray*}
&&H_n({\bf a})=\int_0^\infty 
W_n(x^2;{\bf a})
W_n(x^2;{\bf a})
{\sf W}(x;{\bf a}) dx\\[0.2cm]
&&\hspace{1cm}=\frac{2\pi n!\,
\Gamma(a_{12}+n)
\Gamma(a_{13}+n)
\Gamma(a_{14}+n)
\Gamma(a_{23}+n)
\Gamma(a_{24}+n)
\Gamma(a_{34}+n)
}
{(a_{1234}-1+2n)\Gamma(a_{1234}-1+n)}.
\end{eqnarray*}
\end{cor}
\begin{proof}
Multiply both sides 
of the Wilson polynomial expansion (\ref{Wilsonlimit}) by 
$W_m(x^2;{\bf a}){\sf W}(x;{\bf a})$, 
integrate over $(0,\infty)$ using orthogonality of the Wilson polynomials. 
Replace in resulting expression $m \mapsto n$ and the result follows.
\end{proof}

\subsection{Continuous $q$-Jacobi and Jacobi polynomials}

The orthogonality relation for continuous $q$-Jacobi 
polynomials \cite[(14.10.2)]{Koekoeketal}, after scaling so 
that $q^{\alpha+\frac12}\mapsto \alpha$ and 
$q^{\beta+\frac12}\mapsto \gamma$ is
\[
\int_{-1}^{1}
P_m^{(\alpha,\gamma)}(x\vert q)P_n^{(\alpha,\gamma)}(x\vert q)
\frac{w(x;\alpha,\gamma;q)}{\sqrt{1-x^2}} dx = {2\pi}g_n(\alpha,\gamma;q) \delta_{mn},
\]
where 
\[
w(x; \alpha, \gamma;q):=
\left| \frac{ (e^{2 i \theta};q)_{\infty }} { (\alpha^\frac12 e^{i \theta },-\gamma^\frac12 e^{i \theta } ;q^\frac12)_{\infty }} \right|^2,
\]
and 
\[
g_n(\alpha,\gamma;q):= 
\frac{\alpha^n(1-\alpha\gamma)(q^{\frac12}\alpha,q^{\frac12}\gamma,-q(\alpha \gamma)^{\frac12};q)_n
((\alpha \gamma q)^{\frac12},q(\alpha \gamma )^{\frac12};q)_\infty
}
{(1-q^{2n}\alpha\gamma)(q,\alpha \gamma,-(\alpha\gamma)^\frac12;q)_n
(q,q^\frac12 \alpha, q^\frac12 \gamma,-(\alpha \gamma)^\frac12, -(\alpha \gamma q)^\frac12;q)_\infty}.
\]
 
\begin{cor}
Let $n\in\mathbb N_0$, $x=\cos\theta\in(-1,1)$, $\alpha,\gamma\in(-\frac12,\infty)$,
$d_n(\beta,t,\alpha,\gamma;q)$ defined as in (\ref{3:18}). Then
\[
\int_{-1}^{1}
\frac{(t\beta e^{i\theta},t\beta e^{-i\theta};q)_\infty}
{(te^{i\theta},te^{-i\theta};q)_\infty}
\,P_n^{(\alpha,\gamma)}(x\vert q) 
\frac{w(x;\alpha,\gamma;q)}{\sqrt{1-x^2}} dx 
= {2\pi}g_n(\alpha,\gamma;q) d_n(\beta,t,\alpha,\gamma; q).
\]
\end{cor}
\begin{proof}
Multiply (\ref{3:18}) by 
$w(x;\alpha,\gamma;q)P_n^{(\alpha,\gamma)}(x\vert q)/\sqrt{1-x^2}$ 
and integrate over $(-1,1)$ produces the result.
\end{proof}

\begin{cor} Let $n\in\N_0$, $\alpha,\beta>-1$, $\nu\in\C$, such that $\Re(\alpha+1-\nu)>0$. Then
\[
\int_{-1}^1 (1-x)^{-\nu} P_n^{(\alpha,\beta)}(x) (1-x)^\alpha (1+x)^\beta dx=
\frac{
2^{\alpha+\beta+1-\nu}\Gamma(\alpha+1-\nu)(\nu)_n\Gamma(\beta+1+n)
}{
n!\Gamma(\alpha+\beta+2-\nu+n)
}.
\]
\end{cor}
\begin{proof}
Follows from orthogonality of Jacobi polynomials \cite[(9.8.2)]{Koekoeketal} and (\ref{Jacobi1mnux}).
\end{proof}

\subsection{The continuous $q$-ultraspherical/Rogers and Gegenbauer polynomials}

The property of orthogonality for continuous 
$q$-ultraspherical/Rogers polynomials found 
in Koekoek et al. (2010) \cite[(3.10.16)]{Koekoeketal} 
is given by
\begin{equation}
\label{ortho}
\int_{-1}^1 C_m(x;\beta|q)C_n(x;\beta|q)
\frac{w(x;\beta|q)}{\sqrt{1-x^2}}\, dx=
2\pi\frac{(1-\beta)(\beta,q\beta;q)_\infty 
(\beta^2;q)_n}{(1-\beta q^n)(\beta^2,q;q)_\infty (q;q)_n
}\delta_{mn},
\end{equation}
where $w:(-1,1)\to[0,\infty)$ is the 
weight function defined by
\begin{equation}
\label{weight}
w(x;\beta|q):=
\left|
\frac{(e^{2i\theta};q)_\infty}
{(\beta e^{2i\theta};q)_\infty}
\right|
^2.
\end{equation}
We use this orthogonality relation for 
proofs of the following definite integrals.

\begin{cor}
Let $n\in\mathbb N_0$, $\beta,\gamma\in (-1,1)
\setminus\{0\}$, $0<|q|<1$, $|t|<1$. Then
\begin{eqnarray}
&&\int^1_{-1} \frac{(t\beta e^{i \theta},t\beta 
e^{-i\theta};q)_\infty}{(te^{i\theta},te^{-i
\theta};q)_\infty} C_n(x;\gamma|q)\frac{w(x;\gamma|q)}
{\sqrt{1-x^2}} dx \nonumber\\[0.2cm]
&&\hspace{5cm}=2\pi \frac{(\gamma,\gamma q;q)_\infty
(\beta,\gamma^2;q)_n}{(\gamma^2,q;q)_\infty
(q,q\gamma;q)_n}\qhyp21{\beta\gamma^{-1},\beta q^n}
{\gamma q^{n+1}}{q,\gamma t^2}t^n.
\label{6:41}
\end{eqnarray}
\end{cor}
\begin{proof}We begin with the generalized 
generating function (\ref{4:21}),
multiply both sides by
\[
C_m(x;\gamma|q)\frac{w(x;\gamma|q)}{\sqrt{1-x^2}},
\]
where $w(x;\gamma)$ is obtained from 
(\ref{weight}), integrating over $(-1,1)$ 
using the orthogonality relation (\ref{ortho}),
produces the desired result.
\end{proof}

\begin{cor} Let $n\in{\mathbb N}_0$, 
$\lambda,\mu\in(-\frac12,\infty)\setminus\{0\}$, $|t|<1$. Then
\[
\int_{-1}^1 \dfrac{C_n^\mu(x)}{(1-2tx
+t^2)^\lambda}\, 
(1-x^2)^{\mu-\frac12}
dx
=\dfrac{\sqrt \pi\, \Gamma(\mu+\tfrac12)(\lambda,2\mu)_n}
{\Gamma(\mu+1)(\mu+1)_n n!}\,\hyp 21
{\lambda-\mu, \lambda+n}{\mu+n+1} {t^2} t^n.
\]
\end{cor}
\begin{proof}
Starting from (\ref{6:41}) and taking the limit $q\uparrow 1^{-}$, 
using \cite[(14.10.35)]{Koekoeketal}
\[
\lim_{q\uparrow 1^{-}}C_n(x;q^\lambda|q)=C_n^\mu(x),
\]
the result follows. 
\end{proof}
Observe that since the Gegenbauer polynomials can be written as
\cite[(18.5.10)]{NIST:DLMF}
\[
C_n^\lambda(x)=(2x)^n\dfrac{(\lambda)_n}{n!}
\hyp 21{-\frac{n}{2},-\frac{n+1}{2}}{1-\lambda-n}
{\frac 1{x^2}},
\] 
the above integral can be written in terms of a ${}_2 F_1$, and 
we also have a similar ${}_2 F_1$ on the right-hand side.


\begin{cor} Let $n\in\N_0$, $\alpha,\beta>-1$, $\nu\in\C$, such that $\Re(\alpha+1-\nu)>0$. Then
\[
\int_{-1}^1 (1-x)^{-\nu} C_n^{\mu}(x) (1-x^2)^{\mu-\frac12}dx=
\frac{
2^{\alpha+\beta+1-\nu}\Gamma(\alpha+1-\nu)(\nu)_n\Gamma(\beta+1+n)
}{
n!\Gamma(\alpha+\beta+2-\nu+n)
}.
\]
\end{cor}
\begin{proof}
Follows from orthogonality of Gegenbauer polynomials \cite[(9.8.20)]{Koekoeketal} and (\ref{Geg1mnux}).
\end{proof}

Similar definite integrals can be obtained for the Chebyshev polynomials 
and the first kind multiplied by $(1-x)^{-\nu}$ and for the Laguerre polynomials 
multiplied by $1/x^\nu$, using (\ref{sum1mncCheby}) and (\ref{sum1mncLag}) respectively.

\section{Outlook}

It has been suggested by the referee that it would be interesting to 
investigate the transformation properties of the derived definite integrals
in this paper since the Rogers generating function is a generalization 
of the generalized Steiltjes kernel $(z-x)^{-\nu}$. The transformation
and transmutation properties of the generalized Stieltjes transformations
for the Gauss hypergeometric function has been summarized recently in a 
a paper by Koornwinder \cite{Koornwinder2015}. 
Generalized Stieltjes transforms have evident properties 
of mapping solutions of a hypergeometric 
differential equation to other solutions of 
the same equation, while generalized Stieltjes 
transforms map solutions of a hypergeometric 
differential equation to solutions of another 
differential equation.
Unfortunately a similar 
analysis for our problem is not easily accomplished because the singularities
of the Gauss hypergeometric differential equation are $0$, $1$ and $\infty$,
whereas for instance, for Jacobi-type orthogonal polynomials the singularies
are at $\pm 1$ and $\infty$.  In future research, we would like to apply an 
analogous result to study the transformation properies for definite
integrals of Jacobi-type orthogonal polynomials and also for their $q$-analogs 
such as for continuous $q$-ultraspherical polynomials using the Gegenbauer
and Rogers generating functions. This study could have deep consequences.

\section*{Acknowledgements}
Much thanks to Hans Volkmer, Mourad Ismail, Tom 
Koornwinder, Michael Schlosser, and T.~M.~Dunster for valuable 
discussions.  The author R. S. Costas-Santos acknowledges 
financial support by Direcci\'on General de 
Investigaci\'on, Ministerio de Econom\'ia 
y Competitividad of Spain, grant MTM2015-65888-C4-2-P.


\def\cprime{$'$} \def\dbar{\leavevmode\hbox to 0pt{\hskip.2ex \accent"16\hss}d}

\end{document}